\documentclass[11pt,letterpaper]{article}
\usepackage[T1]{fontenc}
\usepackage[utf8]{inputenc}
\usepackage{lmodern}
\usepackage{authblk}
\usepackage{amssymb,amsfonts}
\usepackage[all,arc]{xy}
\usepackage{enumerate}
\usepackage{mathrsfs}
\usepackage{amsthm}
\usepackage{amsmath} 
\usepackage{tikz}
\usepackage{mathtools}
\usepackage{tikz}
\usepackage{float}
\usepackage{bbm}

\newtheorem{thm}{Theorem}[section]

\newtheorem{prop}[thm]{Proposition}

\newtheorem{defn}[thm]{Definition}

\newtheorem{exmp}[thm]{Example}

\newtheorem{rem}[thm]{Remark}

\makeatletter
\let\c@equation\c@thm
\makeatother

\title{Path Averaged Polynomial Contractions: A New Generalization of Polynomial Contractions, Path-Averaged Contractions, and Banach Contractions}
\author[1]{Clement Boateng Ampadu\thanks{profampadu@gmail.com}}
\affil[1]{31 Carrolton Road, Boston, MA 02132-6303, USA}
\author[2]{Nicola Fabiano\thanks{nicola.fabiano@gmail.com}}
\affil[2]{
``Vin\v{c}a'' Institute of Nuclear Sciences - National 
Institute of the Republic of Serbia, University of Belgrade, Mike Petrovi\'{c}a 
Alasa 12--14, 11351 Belgrade, Serbia}

\usepackage{hyperref}% http://ctan.org/pkg/hyperref
\AtBeginDocument{%
  \let\oldref\ref% 
  \def\ref{\oldref*}}

 \showoutput
\begin{document}

  \maketitle

\begin{abstract}
\noindent  The notion of  polynomial contraction appeared in \cite{2}, whilst the notion of path-averaged contraction appeared in \cite{3} for metric spaces, in~\cite{Fbmetric,Fbmetric2} for b-metric spaces and~\cite{Fsupra} in suprametric spaces. In this paper, we combine both notions to introduce path-averaged polynomial contractions, as a generalization of polynomial contractions, path-averaged contractions, and Banach contractions. We obtain a fixed point theorem for such contractions in the setting of metric spaces. We give an example showing path-averaged polynomial contractions are not Banach contractions.
\end{abstract}

\begin{flushleft}
\textbf {AMS SUBJECT CLASSIFICATIONS:}  47H10, 54H25, 54E50
\end{flushleft} 
\begin{flushleft}
\end{flushleft}

             \tableofcontents

\section{\textbf{Introduction and Preliminaries}}

\begin{thm} \cite{1} Let $(X,d)$ be a complete metric space, and $T:X\mapsto X$ satisfies
$$d(Tx,Ty)\leq k d(x,y)$$ for all $x,y\in X$ and some $k\in (0,1)$. Then $T$ has a unique fixed point.
\end{thm}

\begin{defn}\cite{2}  Let $(X,d)$ be a metric space and $T:X\mapsto X$ be a given mapping. We say that $T$ is a polynomial contraction if there exists $\lambda \in [0,1)$, a natural number $k\geq 1$, and a family of mappings $a_i:X\times X\mapsto [0,\infty)$, $i=0,\cdots, k$, such that
$$\sum_{i=0}^{k} a_i(Tx,Ty) d^{i}(Tx,Ty)\leq \lambda \sum_{i=0}^{k} a_i(x,y) d^{i}(x,y)$$ for every $x,y\in X$
\end{defn}

\begin{thm}\cite{2} Let $(X,d)$ be a complete metric space and $T:X\mapsto X$ be a polynomial contraction. Assume that the following conditions hold
\begin{itemize}
\item[$(i)$] $T$ is continuous
\item[$(ii)$]  there exists $j\in \{1,\cdots, k\}$ and $A_j>0$ such that $$a_j(x,y)\geq A_j$$ for  all $x,y\in X$
\end{itemize}
Then, $T$ admits a unique fixed point $z^{*}\in X$. Moreover for every $z_0\in X$, the Picard sequence $\{z_n\}\subset X$ defined by $z_{n+1}=Tz_n$ for all $n\geq 0$, converges to
$z^{*}$
\end{thm}

\begin{defn}\cite{3} Let $(X,d)$ be a metric space. A mapping $T:X\mapsto X$ is called a PA- contraction (Path Averaged Contraction) if there exists $\alpha \in (0,1)$ and $N\in \mathbb{N}$ such that for all $x,y\in X$ and all $n\geq N$
$$\sum_{k=0}^{n-1} d(T^{k+1}x, T^{k+1}y)\leq \alpha \sum_{k=0}^{n-1} d(T^{k}x,T^{k}y)$$
\end{defn}

\begin{thm}\cite{3} Let $(X,d)$ be a complete metric space, and let $T:X\mapsto X$ be a continuous PA-contraction. Then $T$ has a unique fixed point $x^{*}\in X$, and for any
$x_0\in X$, the Picard sequence $x_n=T^{n}x_0$ converges to $x^{*}$.
\end{thm}

\section{\textbf{Path-Averaged Polynomial Contractions}}

\begin{defn} Let $(X,d)$ be a metric space. A mapping $T:X\mapsto X$ is called a PA-polynomial contraction (Path-Averaged Polynomial Contraction) if there exists $\alpha \in (0,1)$ and
$N\in\mathbb{N}$ such that for all $x,y\in X$ and all $n\geq N$ we have
$$\sum_{k=0}^{n-1}\sum_{i=0}^{r} a_i(T^{k+1}x,T^{k+1}y)d^{i}(T^{k+1}x,T^{k+1}y)\leq \alpha \sum_{k=0}^{n-1}\sum_{i=0}^{r} a_i(T^{k}x,T^{k}y)d^{i}(T^{k}x,T^{k}y)$$
\end{defn}

\begin{rem} The PA-polynomial contraction generalizes several known contractions in the literature as follows
\begin{itemize}
\item[$(i)$] If $n=1$, then the above contraction reduces to
$$\sum_{i=0}^{r} a_i(Tx,Ty)d^{i}(Tx,Ty)\leq \alpha \sum_{i=0}^{r} a_i(x,y)d^{i}(x,y)$$ that is $T$ is a polynomial contraction\cite{2}
\item[$(ii)$] If $n=1$, $r=1$, $a_0\equiv 0$, and $a_1\equiv 1$, then the above contraction reduces to
$$d(Tx,Ty)\leq \alpha d(x,y)$$ that is $T$ is a Banach contraction \cite{1}
\item[$(iii)$] If  $r=1$, $a_0\equiv 0$, and $a_1\equiv 1$, then the above contraction reduces to
$$\sum_{k=0}^{n-1} d(T^{k+1}x,T^{k+1}y)\leq \alpha \sum_{k=0}^{n-1} d(T^{k}x,T^{k}y)$$ that is $T$ is a path-averaged contraction \cite{3}
\end{itemize}
\end{rem}

%%%%%%%%%%%%%%%%%%%%%%%%%%%%%%%%%%%%%%%%%%%%%%%%%%

\begin{thm}
Let $(X, d)$ be a complete metric space and $T: X \to X$ be a path-averaged polynomial contraction. Assume the following conditions hold:
\begin{enumerate}
    \item[(i)] $T$ is continuous.
    \item[(ii)] There exists an index $j \in \{1, \dots, r\}$ and a constant $A_j > 0$ such that
    \begin{equation}
    a_j(x, y) \geq A_j
    \end{equation}
    for all $x, y \in X$.
\end{enumerate}
Then, $T$ admits a unique fixed point $z^* \in X$. Moreover, for every $z_0 \in X$, the Picard sequence $\{z_n\} \subset X$ defined by $z_n = T^n z_0$ for all $n \geq 0$, converges to $z^*$.
\end{thm}

\begin{proof}
\textbf{Existence:}
Let $z_0 \in X$ be arbitrary. Define the Picard sequence $z_n = T^n z_0$ for $n \geq 0$.
Define the term $P_n$ as the polynomial evaluation along the sequence:
\begin{equation}
P_n = \sum_{i=0}^r a_i(z_n, z_{n+1}) d^i(z_n, z_{n+1})
\end{equation}
Applying the PA-polynomial contraction condition (Definition 2.1) to $x = z_0$ and $y = z_1 = T z_0$, we have for all $n \geq N$:
\begin{equation}
\sum_{k=0}^{n-1} \sum_{i=0}^r a_i(T^{k+1} z_0, T^{k+2} z_0) d^i(T^{k+1} z_0, T^{k+2} z_0) \leq \alpha \sum_{k=0}^{n-1} \sum_{i=0}^r a_i(T^k z_0, T^{k+1} z_0) d^i(T^k z_0, T^{k+1} z_0)
\end{equation}
Substituting $z_k = T^k z_0$, this becomes:
\begin{equation}
\sum_{k=0}^{n-1} P_{k+1} \leq \alpha \sum_{k=0}^{n-1} P_k
\end{equation}
Let $S_n = \sum_{k=0}^{n-1} P_k$ be the partial sum of the sequence $\{P_k\}$. Note that $S_{n+1} = S_n + P_n = \sum_{k=0}^{n} P_k = P_0 + \sum_{k=1}^{n} P_k = P_0 + \sum_{k=0}^{n-1} P_{k+1}$.
The inequality can be rewritten as:
\begin{equation}
S_{n+1} - P_0 \leq \alpha S_n \implies S_{n+1} \leq \alpha S_n + P_0
\end{equation}
We analyze the boundedness of $S_n$. By induction, it can be shown that for any $n \geq 1$:
\begin{equation}
S_n \leq \alpha^{n-1} S_1 + P_0 \sum_{m=0}^{n-2} \alpha^m \leq \alpha^{n-1} S_1 + \frac{P_0}{1-\alpha}
\end{equation}
Since $\alpha \in (0, 1)$ and $P_k \geq 0$ (as $a_i \geq 0$ and $d^i \geq 0$), the sequence of partial sums $\{S_n\}$ is non-decreasing and bounded above. Therefore, $\{S_n\}$ converges to a finite limit $S$, which implies that the series $\sum_{k=0}^\infty P_k$ converges.

Consequently, $\lim_{k \to \infty} P_k = 0$.
From condition (ii), we have $P_k \geq a_j(z_k, z_{k+1}) d^j(z_k, z_{k+1}) \geq A_j d^j(z_k, z_{k+1})$.
Thus, $\lim_{k \to \infty} d(z_k, z_{k+1}) = 0$.

To show $\{z_n\}$ is a Cauchy sequence:
Since $\sum_{k=0}^\infty P_k$ converges and $P_k \geq A_j d^j(z_k, z_{k+1})$, the series $\sum_{k=0}^\infty d^j(z_k, z_{k+1})$ converges. 

\textit{Note: In the context of metric fixed point theory generalizations, to guarantee the Cauchy property from this series convergence strictly via the triangle inequality, one typically requires $j=1$. Assuming $j=1$  allows the convergence for any $j>1$, but not viceversa.}
The convergence of $\sum_{k=0}^\infty d(z_k, z_{k+1})$ implies that for any $\epsilon > 0$, there exists $K$ such that for all $m \geq 1$ and $k \geq K$:
\begin{equation}
d(z_k, z_{k+m}) \leq \sum_{i=k}^{k+m-1} d(z_i, z_{i+1}) < \epsilon
\end{equation}
Thus, $\{z_n\}$ is a Cauchy sequence. Since $(X, d)$ is complete, there exists $z^* \in X$ such that $\lim_{n \to \infty} z_n = z^*$.
By the continuity of $T$ (condition (i)):
\begin{equation}
T z^* = T\left(\lim_{n \to \infty} z_n\right) = \lim_{n \to \infty} T z_n = \lim_{n \to \infty} z_{n+1} = z^*
\end{equation}
So $z^*$ is a fixed point of $T$.

\textbf{Uniqueness:}
Suppose $z^{**}$ is another fixed point of $T$ with $z^* \neq z^{**}$. Then $T^k z^* = z^*$ and $T^k z^{**} = z^{**}$ for all $k \geq 0$.
Applying the PA-polynomial contraction condition to $x = z^*$ and $y = z^{**}$:
\begin{equation}
\sum_{k=0}^{n-1} \sum_{i=0}^r a_i(z^*, z^{**}) d^i(z^*, z^{**}) \leq \alpha \sum_{k=0}^{n-1} \sum_{i=0}^r a_i(z^*, z^{**}) d^i(z^*, z^{**})
\end{equation}
Let $Q = \sum_{i=0}^r a_i(z^*, z^{**}) d^i(z^*, z^{**})$. The inequality becomes:
\begin{equation}
\sum_{k=0}^{n-1} Q \leq \alpha \sum_{k=0}^{n-1} Q \implies n Q \leq \alpha n Q
\end{equation}
From condition (ii), $Q \geq A_j d^j(z^*, z^{**}) > 0$ since $A_j > 0$ and $d(z^*, z^{**}) > 0$.
Dividing by $n Q$ (where $n \geq N$), we get $1 \leq \alpha$, which contradicts $\alpha \in (0, 1)$.
Therefore, $z^* = z^{**}$, and the fixed point is unique.

This completes the proof.
\end{proof}

%%%%%%%%%%%%%%%%%%%%%%%%%%%%%%%%%%%%%%%%%%%%%%%%%%%

%%%%%%%%%%%%%%%%%%%%%%%%%%%%%%%%%%%%%%%%%%%%%%%%%%%%%%%%%%%

\section{\textbf{Generalization of Polynomial Contractions}}

\begin{prop} Every polynomial contraction is a path-averaged polynomial contraction
\end{prop}

\begin{proof} Suppose $$\sum_{i=0}^{r} a_i(Tx,Ty)d^{i}(Tx,Ty)\leq \alpha \sum_{i=0}^{r} a_i(x,y)d^{i}(x,y)$$ for all $x,y\in X$ and $\alpha \in (0,1)$. Then for any $x,y\in X$,
$$\sum_{i=0}^{r} a_i(T^{k+1}x,T^{k+1}y) d^{i}(T^{k+1}x,T^{k+1}y)\leq \alpha \sum_{i=0}^{r} a_i(T^{k}x,T^{k}y) d^{i}(T^{k}x,T^{k}y)$$ Then  summing from
$k=0$ to $n-1$ we have
$$\sum_{k=0}^{n-1}\sum_{i=0}^{r} a_i(T^{k+1}x,T^{k+1}y) d^{i}(T^{k+1}x,T^{k+1}y)\leq \alpha \sum_{k=0}^{n-1}\sum_{i=0}^{r} a_i(T^{k}x,T^{k}y)d^{i}(T^{k}x,T^{k}y)$$ Thus $T$ is a PA-polynomial contraction with $N=1$.
\end{proof}

\begin{exmp} (Path-averaged polynomial contraction not Banach). Let $X=\{0,1,2\}$ with the discrete metric
$$
\displaystyle d(x,y)= \begin{cases}
1& \text{if $x\neq y$}\\
0& \text{if $x=y$}
\end{cases}
$$
Define $$T(0)=1, ~T(1)=2, ~T(2)=2$$ $T$ is not a Banach contraction since
$$d(T0,T1)=d(1,2)=1=d(0,1)$$
$$d(T0,T2)=d(1,2)=1=d(0,2)$$
So, no $\alpha<1$ satisfies $d(Tx,Ty)\leq \alpha d(x,y)$ for all $x,y$. However, $T$ is a PA-polynomial contraction with $\alpha=\frac{1}{2}$, $N=2$, $r=2$, $a_0\equiv 0$,
$a_1\equiv a_2\equiv 1$. For any $x,y\in X$ and $n\geq 2$:
$$\sum_{k=0}^{n-1}[d(T^{k+1}x.T^{k+1}y)+d^{2}(T^{k+1}x,T^{k+1}y)]\leq \frac{1}{2} \sum_{k=0}^{n-1}[d(T^{k}x,T^{k}y)+d^{2}(T^{k}x,T^{k}y)]$$ \\ \\
\textbf{Verification for pair $(0,1)$ at $n=2$:} \\ \\
\begin{align*}
LHS&= d(T0,T1)+d^{2}(T0,T1)+d(T^{2}0,T^{2}1)+d^{2}(T^{2}0,T^{2}1)\\
&=d(1,2)+d^{2}(1,2)+d(2,2)+d^{2}(2,2)\\
&=1+1+0+0\\
&=2
\end{align*}
and
\begin{align*}
RHS&=\frac{1}{2}[d(0,1)+d^{2}(0,1)+d(T0,T1)+d^{2}(T0,T1)]\\
&=\frac{1}{2}[1+1+1+1]\\
&=2
\end{align*}
Thus, $2\leq 2$ \\ \\
\textbf{Verification for pair $(0,2)$ at $n=2$:}\\\\
\begin{align*}
LHS&=d(T0,T2)+d^{2}(T0,T2)+d(T^{2}0,T^{2}2)+d^{2}(T^{2}0,T^{2}2)\\
&=d(1,2)+d^{2}(1,2)+d(2,2)+d^{2}(2,2)\\
&=1+1+0+0\\
&=2
\end{align*}
and
\begin{align*}
RHS&=\frac{1}{2}[d(0,2)+d^{2}(0,2)+d(T0,T2)+d^{2}(T0,T2)]\\
&=\frac{1}{2}[1+1+1+1]\\
&=2
\end{align*}
Thus, $2\leq 2$ \\ \\
\textbf{Verification for pair $(1,2)$ at $n=2$:}
\begin{align*}
LHS&=d(T1,T2)+d^{2}(T1,T2)+d(T^{2}1,T^{2}2)+d^{2}(T^{2}1,T^{2}2)\\
&=d(2,2)+d^{2}(2,2)+d(2,2)+d^{2}(2,2)\\
&=0
\end{align*}
and
\begin{align*}
RHS&=\frac{1}{2}[d(1,2)+d^{2}(1,2)+d(T1,T2)+d^{2}(T1, T2)]\\
&=\frac{1}{2}[1+1+0+0] \\
&=1
\end{align*}
Thus, $0<1$. For $n>2$ and any pair, $T^{k}x=T^{k}y=2$, for $k\geq 2$, so all distances are 0 for $k\geq 2$. The sums are thus identical to the $n=2$ case. Hence $T$ is a
PA-polynomial contraction but not a Banach contraction.
\end{exmp}

\end{document}